\documentclass[12pt]{article}
\usepackage[utf8]{inputenc}
\usepackage[english]{babel}
\usepackage{amsmath,amssymb,amsthm,mathtools}
\usepackage{indentfirst}
\usepackage{tikz,float}
\usepackage[colorlinks=true,allcolors=blue]{hyperref}

\newtheorem{theorem}{Theorem}
\newtheorem{lemma}{Lemma}
\newtheorem{corollary}{Corollary}
\newcommand{\R}{\mathbb{R}}
\renewcommand{\P}{\mathbb{P}}
\newcommand{\eps}{\varepsilon}

\renewcommand{\leq}{\leqslant}
\renewcommand{\geq}{\geqslant}
\begin{document}
\title{An improved upper bound on the chromatic number of the $n$-dimensional sphere}
\author{A. I. Bikeev, V. A. Voronov, A. M. Raigorodskii}
\date{}
\maketitle

\begin{abstract}
    In this short note, we obtain upper bounds on the chromatic numbers of spheres using recent improvements in bounds for Euclidean space. We show that $\chi (S^n_{R}) \leq (\xi_R + o(1))^n$, where $\xi_R = \frac{3\sqrt 3}{2}(1-\frac{1}{4R^2}) $ and $\chi(S^n_R)$ denotes the chromatic number of the $n$-dimensional sphere of a fixed radius $R$.
\end{abstract}

\section{Introduction}

In this short note, we obtain upper bounds on the chromatic numbers of spheres and balls using recent improvements in bounds for Euclidean space. We also extend these bounds to forbidden intervals of distances and to sets of prescribed diameter. Recall that the chromatic number of a subset $A\subseteq\R^n$ is the minimum number $\chi(A)$ of colors required to color its points so that any two points at Euclidean distance exactly $1$ have different colors.

For a long time, the best asymptotic bounds for Euclidean space were
\begin{equation}\label{eq:1239}
 (1.239+o(1))^n\leq\chi(\R^n)\leq(3+o(1))^n.
\end{equation}
The lower bound is due to Raigorodskii~\cite{Rai00}, and the upper bound to Larman and Rogers~\cite{LR}. A recent paper by Davies~\cite{Davies} gives
\begin{equation}\label{eq:new}
 (2+o(1))^n\leq\chi(\R^n)\leq\left(\frac{3\sqrt3}{2}+o(1)\right)^n.
\end{equation}

Let $S_R^n=\{x\in\R^{n+1}:\|x\|=R\}$; in particular, $S_1^n$ is the unit sphere. All asymptotic statements below are for fixed $R$ as $n\to\infty$. Clearly, $\chi(S_R^{n-1})\leq\chi(\R^n)$, so \eqref{eq:1239} and \eqref{eq:new} also give upper bounds for spheres. Raigorodskii~\cite{Rai12} proved exponential growth for every fixed $R>1/2$, and the lower bound
\[
 \chi(S_R^{n-1})\geq\bigl(2q^q(1-q)^{1-q}+o(1)\bigr)^n,
 \qquad q=\frac{1}{8R^2},
\]
for $1/2<R<1/\sqrt2$. Kostina~\cite{Kostina} sharpened lower bounds using constant-composition polynomial constructions (see also a recent paper \cite{Raigorodskii_Sagdeev}). Covering by spherical caps gives the classical upper bound $(2R+o(1))^n$ for $\chi(S_R^n)$; see~\cite{Ro63, Naszodi,Prosanov} (a huge number of works are related to refinement of the infinitesimal term in the base of the exponent; see, for example, \cite{Ver, BW}). Prosanov~\cite{Prosanov} obtained the following improvement.

\begin{theorem}\label{thm:prosanov}
For every fixed $R>1/2$,
\begin{equation}\label{eq:prosanov}
 \chi(S_R^n)\leq(x(R)+o(1))^n,
\end{equation}
where
\begin{equation}\label{eq:x}
 x(R)=\begin{cases}
 \sqrt{5-\dfrac{2}{R^2}+4\sqrt{1-\dfrac{5}{4R^2}+\dfrac{1}{4R^4}}},
 &R>\dfrac{\sqrt5}{2},\\[4pt]
 2R,&\dfrac12<R\leq\dfrac{\sqrt5}{2}.
 \end{cases}
\end{equation}
\end{theorem}

Prosanov's proof contracts spherical Voronoi cells and covers the sphere by rotated copies of the resulting set. The base $x(R)$ tends to $3$ as $R\to\infty$, corresponding to the Larman--Rogers bound in~\eqref{eq:1239}. We obtain a bound whose base tends to the constant in Davies's upper bound~\eqref{eq:new}.

Put
\[
 \xi_R=\frac{3\sqrt3}{2}\left(1-\frac{1}{4R^2}\right).
\]
The main result is the following.
\begin{theorem}\label{main}
For every fixed $R>1/2$,
\begin{equation}\label{eq:main}
 \chi(S_R^n)\leq
 \begin{cases}
 (\xi_R+o(1))^n, & R > \sqrt{3}/2\\
 (2R+o(1))^n, & 1/2 < R \leq \sqrt{3}/2
 \end{cases}
\end{equation}
\end{theorem}
The base $\xi_R$ tends to $3\sqrt3/2$ as $R\to\infty$ and is smaller than $x(R)$ for every $R>\sqrt3/2$.

\medskip
\emph{Sketch of the proof.}
Following~\cite{Davies}, we choose independent random cap centers. The caps have diameter $1-\eps$ and may overlap. We trim them so that distances between distinct retained pieces exceed $1+\eps$. Their union $P$ therefore avoids every distance in $[1-\eps,1+\eps]$. We cover an $\eps/4$-net by sufficiently many random rotations of $P$. The covering succeeds with positive probability for a suitable number of copies.
The $\eps/3$-neighborhoods of these copies cover the sphere and avoid unit distance, giving a proper coloring. Spherical geometry requires an additional probability estimate, while compactness removes the need for the periodic construction used in Euclidean space.

Note that instead of \(\varepsilon/4\) and \(\varepsilon/3\), one can take any constants \(\delta_1, \delta_2\) satisfying \(0<\delta_1<\delta_2<\varepsilon/2\). This does not affect the  estimate.
\medskip

Let $B_R^n$ be the closed ball of radius $R$ in $\R^n$. Since, for fixed \(R\) and \(n \to \infty\), almost all of the volume of $B_R^n$ is concentrated near its boundary, it is quite natural that the asymptotics of the chromatic numbers of the ball can be estimated by the same formula.
\begin{corollary}\label{cor:1}
For every fixed $R>1/2$,
\[
 \chi(B_R^n)\leq
  \begin{cases}
 (\xi_R+o(1))^n, & R > \sqrt{3}/2\\
 (2R+o(1))^n, & 1/2 < R \leq \sqrt{3}/2
 \end{cases}
\]
\end{corollary}

The proof repeats the analogous spherical-shell argument of~\cite{Prosanov} in full, and we do not reproduce it.

\section{Notation and auxiliary statements}

Write $x\cdot y$ for the scalar product in $\R^{n+1}$. For $0\leq r\leq1/\sqrt3$, put
\[
 \eta(r)=\frac{3\sqrt3}{2}r(1-r^2).
\]
This function is increasing on the stated interval, with $\eta(0)=0$ and $\eta(1/\sqrt3)=1$. Write
\[
C(z,\varphi)=\{x\in S_1^n:\angle(x,z)\leq\varphi\}
\] 
for the cap centered at $z$. Here $r=\sin\varphi$ is the Euclidean radius of the boundary of a cap of angular radius $\varphi$ on the unit sphere; the corresponding radius on $S_R^n$ is $Rr$.

For fixed $v,u\in S_1^n$ and a uniformly distributed random point $w\in S_1^n$, define
\[
 A^r_{u,v}(w)=\left\{v\cdot\frac{u-w}{\|u-w\|}\leq r\right\}.
\]
This event means that $v$ fails the strict separating inequality associated with the cap centers $u$ and $w$; see the construction in Lemma~\ref{l:1} and Fig.~\ref{fig:lemma-geometry}. The event $w=u$ has probability zero and is omitted.

\begin{lemma}\label{lemma2}
Let $n\geq2$ and $0\leq r\leq1/\sqrt3$. If $v,u\in S_1^n$ satisfy $\angle(v,u)\leq\varphi=\arcsin r$, then
\begin{equation}\label{eq:probability}
 \P\bigl(A^r_{u,v}(w)\bigr)\leq\eta(r)^{n-1}.
\end{equation}
\end{lemma}

\begin{proof}
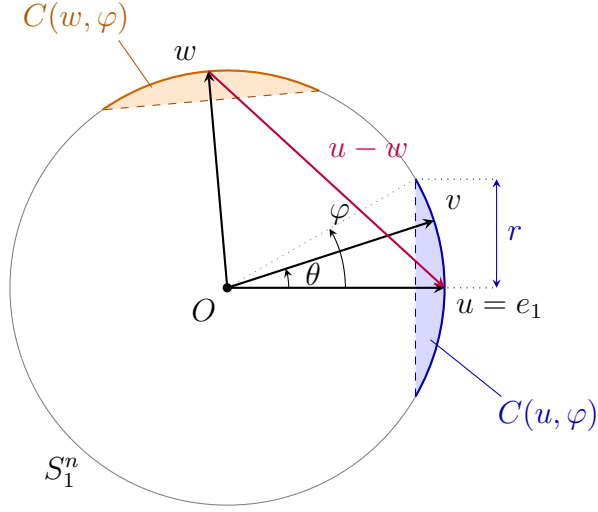
\begin{figure}[H]
\centering
\begin{tikzpicture}[scale=1.25,>=stealth]
 \coordinate (O) at (0,0);
 \coordinate (U) at (0:2.3);
 \coordinate (V) at (18:2.3);
 \coordinate (W) at (95:2.3);
 \fill[blue!15] (-30:2.3) arc[start angle=-30,end angle=30,radius=2.3] -- cycle;
 \fill[orange!20] (65:2.3) arc[start angle=65,end angle=125,radius=2.3] -- cycle;
 \draw[gray] (O) circle[radius=2.3];
 \draw[blue!70!black,thick] (-30:2.3) arc[start angle=-30,end angle=30,radius=2.3];
 \draw[orange!80!black,thick] (65:2.3) arc[start angle=65,end angle=125,radius=2.3];
 \draw[blue!50!black,dashed] (-30:2.3) -- (30:2.3);
 \draw[orange!60!black,dashed] (65:2.3) -- (125:2.3);
 \draw[gray,dotted] (O) -- (30:2.3);
 \draw[->,thick] (O) -- (U) node[below right] {$u=e_1$};
 \draw[->,thick] (O) -- (V) node[above right] {$v$};
 \draw[->,thick] (O) -- (W) node[above left] {$w$};
 \draw[->,thick,purple] (W) -- (U)
   node[pos=0.46,above right] {$u-w$};
 \draw[->] (0:0.65) arc[start angle=0,end angle=18,radius=0.65];
 \node at (9:0.92) {$\theta$};
 \draw[->] (0:1.25) arc[start angle=0,end angle=30,radius=1.25];
 \node at (33:1.42) {$\varphi$};
 \draw[blue!50!black,dotted] (1.992,0) -- (2.85,0);
 \draw[blue!50!black,dotted] (1.992,1.15) -- (2.85,1.15);
 \draw[<->,blue!70!black] (2.85,0) -- (2.85,1.15)
   node[midway,right] {$r$};
 \fill (O) circle[radius=1.4pt] node[below left] {$O$};
 \node at (-1.75,-1.95) {$S_1^n$};
 \node[blue!70!black] at (3.4,-1.35) {$C(u,\varphi)$};
 \draw[blue!70!black,thin] (2.9,-1.15) -- (2.15,-0.55);
 \node[orange!70!black] at (-1.6,2.85) {$C(w,\varphi)$};
 \draw[orange!70!black,thin] (-1.15,2.62) -- (-0.8,2.12);
\end{tikzpicture}
\caption{Schematic planar section of the caps in Lemma~\ref{lemma2}.}
\label{fig:lemma-geometry}
\end{figure}
For $r=0$ the probability is zero, and for $r=1/\sqrt3$ the bound is immediate. Assume $0<r<1/\sqrt3$, and put $h=\sqrt{1-r^2}$ and $\theta=\angle(v,u)$. Choose coordinates with $u=e_1$ and $v=\cos\theta\,e_1+\sin\theta\,e_2$. Write $w_1,w_2$ for the first two coordinates of $w$, and put
\[
 t=\sqrt{\frac{1-w_1}{2}},\qquad \|u-w\|=2t.
\]
When $\theta>0$, the event $A^r_{u,v}(w)$ implies
\[
 w_2\geq2t\frac{t\cos\theta-r}{\sin\theta},
\]
and hence
\begin{equation}\label{eq:E}
 1-w_1^2-w_2^2\leq F(\theta,t)
 :=4t^2\left[1-t^2-
 \left(\frac{\max(t\cos\theta-r,0)}{\sin\theta}\right)^2\right].
\end{equation}
Here $F(\theta,t)$ is nondecreasing in $\theta$, and $F(\theta,t)\leq F(\varphi,t)$. With $a=r/h$,
\[
 F(\varphi,t)=\begin{cases}
 4t^2(1-t^2),&0\leq t\leq a,\\[2pt]
 4t^3(2rh-t)/r^2,&a\leq t\leq1.
 \end{cases}
\]
Since $a<1/\sqrt2$ and $a<3rh/2<1$, maximizing the two branches gives $t_*=3rh/2$ and 
\[
 \max_{0\leq t\leq1}F(\varphi,t)
 =F(\varphi,t_*)=\frac{27}{4}r^2(1-r^2)^2=\eta(r)^2.
\]
If $\theta=0$, the event is simply $t\leq r$, so
\[
 1-w_1^2-w_2^2\leq4t^2(1-t^2)
 \leq4r^2(1-r^2)\leq\eta(r)^2.
\]
Finally, $w_1^2+w_2^2$ has distribution $\operatorname{Beta}(1,(n-1)/2)$, giving
\[
 \P(w_1^2+w_2^2\geq1-s^2)=s^{n-1},\qquad0\leq s\leq1.
\]
In both cases $A^r_{u,v}$ implies $w_1^2+w_2^2\geq1-\eta(r)^2$. Substituting $s=\eta(r)$, we have~\eqref{eq:probability}.
\end{proof}

\begin{lemma}\label{l:1}
Fix $R>\sqrt3/2$ and put $\eps=n^{-2}$. For all sufficiently large $n$, there is a measurable set $P\subset S_R^n$ avoiding every distance in $[1-\eps,1+\eps]$ whose normalized surface measure $p$ satisfies
\[
 p^{-1}\leq(\xi_R+o(1))^n.
\]
\end{lemma}

\begin{proof}
Put
\[
 r=\frac{1}{2R},\qquad r_\pm=r(1\pm\eps),\qquad
 \varphi_-=\arcsin r_-.
\]
For all sufficiently large $n$, $r_+<1/\sqrt3$. By Lemma~\ref{lemma2}, for $v,u\in S_1^n$ with $\angle(v,u)\leq\varphi_-$ and uniform $w\in S_1^n$,
\[
 \P\left(v\cdot\frac{u-w}{\|u-w\|}\leq r_+\right)
 \leq q:=\eta(r_+)^{n-1}.
\]
Choose $N=\lceil1/(2q)\rceil$ independent uniform directions $u_1,\ldots,u_N\in S_1^n$. Retain the  pieces
\[
 P_i=\left\{Rv:\ v\in S_1^n,\ \angle(v,u_i)<\varphi_-,\quad
 Rv\cdot\frac{u_i-u_j}{\|u_i-u_j\|}>\frac{1+\eps}{2}
 \text{ for all }j\ne i\right\}.
\]
Within a piece, distances are less than $2R\sin\varphi_-=1-\eps$. For $x\in P_i$ and $y\in P_j$ with $i\ne j$, the two defining inequalities give
\[
 (x-y)\cdot\frac{u_i-u_j}{\|u_i-u_j\|}>1+\eps.
\]
Thus the pieces are disjoint and $P=\bigcup_iP_i$ avoids every distance in $[1-\eps,1+\eps]$~(Fig.~\ref{fig:trimmed-disks}).

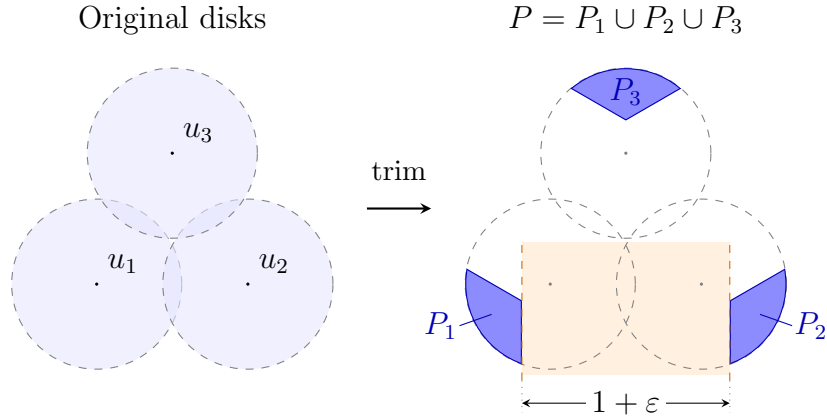
\begin{figure}[htbp]
\centering
\begin{tikzpicture}[x=2.5cm,y=2.5cm,>=stealth]
 \node at (0,1.4) {Original disks};
 \node at (2.4,1.4) {$P=P_1\cup P_2\cup P_3$};
 \foreach \x/\y/\i in {-0.4/0/1,0.4/0/2,0/0.692820/3}{
   \fill[blue!10,fill opacity=0.6] (\x,\y) circle[radius=0.45];
   \draw[gray,dashed] (\x,\y) circle[radius=0.45];
   \fill (\x,\y) circle[radius=0.6pt] node[above right] {$u_{\i}$};
 }
 \draw[->,thick] (1.03,0.4) -- (1.37,0.4);
 \node[font=\small] at (1.2,0.6) {trim};
 \begin{scope}[shift={(2.4,0)}]
 \fill[orange!12] (-0.55,-0.48) rectangle (0.55,0.22);
 \foreach \x/\y in {-0.4/0,0.4/0,0/0.692820}{
   \draw[gray,dashed] (\x,\y) circle[radius=0.45];
   \fill[gray] (\x,\y) circle[radius=0.6pt];
 }
\begin{scope}
 \clip (-0.400000,0.000000) circle[radius=0.45];
 \clip (-2.000000,-2.000000) -- (-0.550000,-2.000000) -- (-0.550000,-0.086603) -- (-2.000000,0.750555) -- cycle;
 \fill[blue!45] (-0.400000,0.000000) circle[radius=0.45];
 \draw[blue!70!black,thick] (-0.400000,0.000000) circle[radius=0.45];
 \draw[blue!70!black,thick] (-2.000000,-2.000000) -- (-0.550000,-2.000000) -- (-0.550000,-0.086603) -- (-2.000000,0.750555) -- cycle;
\end{scope}
\begin{scope}
 \clip (0.400000,0.000000) circle[radius=0.45];
 \clip (0.550000,-2.000000) -- (2.000000,-2.000000) -- (2.000000,0.750555) -- (0.550000,-0.086603) -- cycle;
 \fill[blue!45] (0.400000,0.000000) circle[radius=0.45];
 \draw[blue!70!black,thick] (0.400000,0.000000) circle[radius=0.45];
 \draw[blue!70!black,thick] (0.550000,-2.000000) -- (2.000000,-2.000000) -- (2.000000,0.750555) -- (0.550000,-0.086603) -- cycle;
\end{scope}
\begin{scope}
 \clip (0.000000,0.692820) circle[radius=0.45];
 \clip (1.964102,2.000000) -- (-1.964102,2.000000) -- (0.000000,0.866025) -- cycle;
 \fill[blue!45] (0.000000,0.692820) circle[radius=0.45];
 \draw[blue!70!black,thick] (0.000000,0.692820) circle[radius=0.45];
 \draw[blue!70!black,thick] (1.964102,2.000000) -- (-1.964102,2.000000) -- (0.000000,0.866025) -- cycle;
\end{scope}
 \draw[orange!80!black,dashed] (-0.55,-0.51) -- (-0.55,0.22);
 \draw[orange!80!black,dashed] (0.55,-0.51) -- (0.55,0.22);
 \draw[gray,dotted] (-0.55,-0.51) -- (-0.55,-0.68);
 \draw[gray,dotted] (0.55,-0.51) -- (0.55,-0.68);
 \draw[<->] (-0.55,-0.63) -- (0.55,-0.63) node[midway,fill=white,inner sep=2pt] {$1+\varepsilon$};
 \node[blue!70!black] at (-0.98,-0.22) {$P_1$};
 \draw[blue!70!black,thin] (-0.88,-0.22) -- (-0.71,-0.16);
 \node[blue!70!black] at (0.98,-0.22) {$P_2$};
 \draw[blue!70!black,thin] (0.88,-0.22) -- (0.71,-0.16);
 \node[blue!70!black] at (0,1.02) {$P_3$};
 \end{scope}
\end{tikzpicture}
\caption{The union of trimmed caps.}
\label{fig:trimmed-disks}
\end{figure}

For fixed $Rv\in S_R^n$, condition on $u_i$ with $\angle(v,u_i)<\varphi_-$. The union bound gives conditional survival probability at least $1-(N-1)q\geq1/2$. If $\Omega_n(\varphi_-)$ denotes normalized cap measure, disjointness yields
\[
 \P(Rv\in P)\geq\frac{N\Omega_n(\varphi_-)}{2}
 \geq\frac{\Omega_n(\varphi_-)}{4q}.
\]
Integrating over $v$ shows that some realization has normalized measure at least $\Omega_n(\varphi_-)/(4q)$.

For $0<s<1$, spherical coordinates give
\[
 \Omega_n(\arcsin s)
 =\frac{\Gamma((n+1)/2)}{\sqrt\pi\,\Gamma(n/2)}
 \int_0^s\frac{z^{n-1}}{\sqrt{1-z^2}}\,dz
 \geq\frac{\Gamma((n+1)/2)}{n\sqrt\pi\,\Gamma(n/2)}s^n.
\]
The prefactor is $(1+o(1))/\sqrt{2\pi n}$ by Stirling's formula. Therefore the chosen realization, of measure $p$, satisfies
\[
 p^{-1}\leq O_R(\sqrt n)\left(\frac{\eta(r_+)}{r_-}\right)^n
 =\left(\frac{\eta(r)}{r}+o(1)\right)^n
 =(\xi_R+o(1))^n.
\]
\end{proof}

\section{Proof of Theorem~\ref{main}}

\begin{proof}
For $1/2 \leq R\leq \sqrt3\,t/2$, the required asymptotics easily follows from the classical Rogers's bound \cite{Ro63}.

Fix $R>\sqrt3/2$, and use $\eps=n^{-2}$ and the set $P$ from Lemma~\ref{l:1}. A maximal subset $T\subset S_R^n$ whose distinct points have distance greater than $\eps/4$ is an $\eps/4$-net. Comparing the disjoint ambient balls of radius $\eps/8$ centered at its points with the ball of radius $R+\eps/8$ gives
\[
 |T|\leq\left(1+\frac{8R}{\eps}\right)^{n+1},
 \qquad \log|T|=O_R(n\log n).
\]
Let $p$ be the measure of $P$, and put
\[
 k=\left\lceil\frac{\log|T|+1}{p}\right\rceil.
\]
Take $k$ independent rotations of $P$, sampled from the Haar measure on $SO(n+1)$. Each fixed point is covered with probability $p$ by each copy, so the probability that some point of $T$ is missed is at most
\[
 |T|(1-p)^k\leq|T|e^{-kp}\leq e^{-1}<1.
\]
Choose copies covering $T$, and replace each by its open $\eps/3$-neighborhood in $S_R^n$, using Euclidean distance. These neighborhoods cover the sphere. If two points in the same neighborhood had distance $1$, two points of the original copy would have distance in $(1-2\eps/3,1+2\eps/3)$, a contradiction. Assign each point the color of the first neighborhood containing it. This is a proper coloring, and
\[
 \chi(S_R^n)\leq k\leq O_R(n\log n)\,p^{-1}+1
 \leq(\xi_R+o(1))^n.
\]
\end{proof}

\section{Bounds for related coloring and partition problems}

For a subset $X$ of Euclidean space and $0<b\leq b'$, let
$\chi(X,[b,b'])$ be the minimum number of colors needed to color $X$ so that no color class contains two points whose distance belongs to $[b,b']$.

For $0<b\leq1$, let $\chi(n,b)$ be the least integer $k$ such that every set of diameter $1$ in $\R^n$ can be partitioned into at most $k$ sets of diameter strictly less than $b$. Define $\tilde\chi(n,b)$ similarly, requiring instead that each part avoids distance $b$. For $0<b\leq b'\leq1$, let $\tilde\chi(n,[b,b'])$ denote the corresponding quantity when all distances in $[b,b']$ are forbidden.

In the papers \cite{BR26} and \cite{BR2}, we give a detailed survey of the known asymptotic bounds on these quantities. See also \cite{Kup1} and \cite{Kup2} for some related results.

\subsection{Bounds for \texorpdfstring{$\chi(\R^n,[1,t])$}{Euclidean space}}

For every fixed $t\geq1$, the method of Larman and Rogers~\cite{LR} extends the upper bound in~\eqref{eq:1239} to
\begin{equation}\label{eq:2t1}
 \chi(\R^n,[1,t])\leq(2t+1+o(1))^n.
\end{equation}
Davies's construction~\cite{Davies} also extends to forbidden intervals.

\begin{theorem}\label{t:Davies_for_intervals}
For every fixed $t\geq1$,
\begin{equation}\label{eq:davies_for_intervals}
 \chi(\R^n,[1,t])\leq\left(\frac{3\sqrt3}{2}t+o(1)\right)^n.
\end{equation}
\end{theorem}

\begin{proof}
In Davies's construction~\cite{Davies}, put $\eps=n^{-1/2}$ and replace the open balls of radius $r=(1-\eps)/2$ by balls of radius
$r_1=(1-\eps)/(2t)$. Keep the separating planes, which ensure that distinct retained pieces are at distance greater than $1+\eps$. The resulting union avoids $[(1-\eps)/t,1+\eps]$.

The same bound on the volume of the region of centers that remove a fixed point remains valid, since $r_1\leq r$. The lower bound on the probability that a fixed point is retained is therefore multiplied by $(r_1/r)^n=t^{-n}$. Use a finer periodic grid and enlarge the retained sets by $\eps/(3t)$. These neighborhoods avoid $[1/t,1]$ and cover the space with positive probability. Refining the grid by a fixed factor changes only the subexponential covering factor. Finally, dilate by $t$ to obtain a coloring avoiding $[1,t]$. This proves~\eqref{eq:davies_for_intervals}, including $t=1$.
\end{proof}

The bound~\eqref{eq:davies_for_intervals} is stronger than~\eqref{eq:2t1} for
$1\leq t<t_0$, where
\[
 t_0=\left(\frac{3\sqrt3}{2}-2\right)^{-1}=1.672027\ldots.
\]
The two bases agree at $t=t_0$, and~\eqref{eq:2t1} is stronger for $t>t_0$.

The reason why the linear coefficient of the parameter $t$ in the bounds (8) and (9) differs is as follows. The bound (8) is obtained from the classical argument of Larman and Rogers by contracting the cells of some Voronoi diagram, into which the space is partitioned, by a factor of $2t+1$ instead of $3$. Then the distances between resulting bodies and their diameters are in the ratio at least $t$. On the other hand, in the proof of (9), we contract the original balls of diameter $1$ by a factor of $t$.

\subsection{Bounds for \texorpdfstring{$\chi(B_R^n,[1,t])$}{balls}}\label{ssec:bnr}

In the paper \cite{BR2}, we generalized Prosanov's construction to obtain, for fixed $t\geq1$ and $R>t/2$,
\[
 \chi(B_R^n,[1,t])\leq(\kappa(R,t)+o(1))^n,
\]
where, writing $A=t^2+t$,
\begin{equation}\label{eq:brnt1}
 \kappa(R,t)=\begin{cases}
 \sqrt{1+A\left(2-\dfrac{A}{2R^2}
       +2\sqrt{1-\dfrac{2A+1}{4R^2}+\dfrac{A^2}{16R^4}}\right)},
       &R>\dfrac{\sqrt{2A+1}}2,\\[6pt]
 2R,&\dfrac t2<R\leq\dfrac{\sqrt{2A+1}}2.
 \end{cases}
\end{equation}
We obtain the following analogue of Theorem~\ref{main}.

\begin{theorem}\label{thm:ball-interval}
For fixed $t\geq1$ and $R>1/2$,
\begin{equation}\label{eq:brnt2}
 \chi(B_R^n,[1,t])\leq
 \begin{cases}
 \left(\frac{3\sqrt3}{2}t\left(1-\frac{t^2}{4R^2}\right)+o(1)\right)^n,& R>\sqrt3\,t/2,\\
 (2R+o(1))^n,& 1/2 \leq R\leq \sqrt3\,t/2
 \end{cases}
\end{equation}
\end{theorem}

\begin{proof}
For $1/2 \leq R\leq \sqrt3\,t/2$, this result easily follows from the classical Rogers's bound \cite{Ro63}.
    Let $R > \sqrt{3}\,t/{2}$. In the proof of Lemma \ref{l:1}, replace $r_+ = r(1+\eps)$ by $r_+' = rt$. Since $1 \leq t < \frac{2R}{\sqrt3}$, the condition $r_+' < 1/\sqrt3$ is satisfied, and Lemma \ref{lemma2} remains applicable. The condition $R > \frac{\sqrt{3}}{2}$ also holds. Substituting $\frac{1}{2}$ instead of $\frac{1+\eps}{2}$ in the definition of the set $P_i$, we obtain a set that avoids every distance in $[1-\eps, t]$. Since $\eta(r_+') = \frac{3\sqrt3}{2}rt(1-t^2r^2)$, then the required asymptotics follows.
\end{proof}

Since $\chi(B_R^n,[1,t])\leq\chi(\R^n,[1,t])$, the bounds
\eqref{eq:2t1} and~\eqref{eq:davies_for_intervals} also apply to balls.
For fixed $t$, the bases in~\eqref{eq:brnt1} and~\eqref{eq:brnt2} tend to
$2t+1$ and $3\sqrt3\,t/2$, respectively, as $R\to\infty$.
Their relative strength can depend on $R$: for example, when $t=2$,
\eqref{eq:brnt2} is stronger near $R=\sqrt3$, whereas~\eqref{eq:brnt1} is stronger for sufficiently large $R$. Figure~\ref{fig:int_balls} compares the bases within their stated domains.

\subsection{Bounds for \texorpdfstring{$\tilde\chi(n,b)$}{one forbidden distance}}

In the papers \cite{BR26} and \cite{BR2}, we proved that, for every fixed $b\in(0,1]$,
\begin{equation}\label{eq:tilde_chi_n_b_1}
 \tilde\chi(n,b)\leq\chi(n,b)\leq
 \left(\sqrt{\frac{1}{b^2}+\frac12}+o(1)\right)^n.
\end{equation}
By Jung's theorem, every set of diameter $1$ in $\R^n$ is contained in a ball of radius $\sqrt{n/(2n+2)}<1/\sqrt2$. Scaling by $1/b$ and applying Corollary~\ref{cor:1} and the classical result of Rogers \cite{Ro63} gives
\begin{equation}\label{eq:tilde_chi_n_b_2}
 \tilde\chi(n,b)\leq\chi(B_{1/(b\sqrt2)}^n)\leq
 \begin{cases}
 \left(\dfrac{3\sqrt3}{2}\left(1-\dfrac{b^2}{2}\right)+o(1)\right)^n,
       &0<b<\sqrt{2/3},\\[6pt]
 \left(\dfrac{\sqrt2}{b}+o(1)\right)^n,
       &\sqrt{2/3}\leq b\leq1.
 \end{cases}
\end{equation}
The bases in~\eqref{eq:tilde_chi_n_b_1} and~\eqref{eq:tilde_chi_n_b_2} agree at $b_0=0.449151\ldots$, the unique solution in $(0,\sqrt{2/3})$ of
\[
 \frac{3\sqrt3}{2}\left(1-\frac{b^2}{2}\right)
 =\sqrt{\frac{1}{b^2}+\frac12}.
\]
As a bound on $\tilde\chi(n,b)$, \eqref{eq:tilde_chi_n_b_2} is stronger for $0<b<b_0$, and~\eqref{eq:tilde_chi_n_b_1} is stronger for $b_0<b\leq1$; see Figure~\ref{fig_1_dist}.

\subsection{Bounds for \texorpdfstring{$\tilde\chi(n,[b,b'])$}{forbidden intervals}}

For fixed $0<b\leq b'\leq1$, the bound~\eqref{eq:tilde_chi_n_b_1} gives
\begin{equation}\label{eq:int_bl}
 \tilde\chi(n,[b,b'])\leq\chi(n,b)\leq
 \left(\sqrt{\frac{1}{b^2}+\frac12}+o(1)\right)^n.
\end{equation}
Combining~\eqref{eq:brnt1} with Jung's theorem gives the following estimate.

\begin{theorem}
For fixed $0<b\leq b'\leq1$ and $t=b'/b$,
\[
 \tilde\chi(n,[b,b'])\leq(c(b,t)+o(1))^n,
\]
where, with $A=t^2+t$,
\begin{equation}\label{eq:int_pros}
 c(b,t)=\begin{cases}
 \sqrt{1+A\left(2-b^2A
       +2\sqrt{1-\dfrac{b^2(2A+1)}2+\dfrac{b^4A^2}4}\right)},
       &0<b<\dfrac{\sqrt2}{\sqrt{2A+1}},\\[6pt]
 \dfrac{\sqrt2}{b},
       &\dfrac{\sqrt2}{\sqrt{2A+1}}\leq b\leq\dfrac1t.
 \end{cases}
\end{equation}
\end{theorem}
\begin{proof}
Contain the set in a ball of radius $1/\sqrt2$ and scale by $1/b$.
Then apply~\eqref{eq:brnt1} with $R=1/(b\sqrt2)>t/2$ to obtain
$c(b,t)=\kappa(1/(b\sqrt2),t)$.
\end{proof}

Similarly, Theorem~\ref{thm:ball-interval} gives the following bound.

\begin{theorem}
For fixed $0<b\leq b'\leq 1$, put $t=b'/b$. Then
\begin{equation}\label{eq:int_new}
 \tilde\chi(n,[b,b'])\leq\chi(B_{1/\sqrt2}^n,[b,b'])\leq
 \begin{cases}
 \left(\frac{3\sqrt3}{2}t\left(1-\frac{t^2b^2}{2}\right)+o(1)\right)^n, & t \leq \sqrt{2/3},
 \\
 (\sqrt{2}/b + o(1))^n, & \sqrt{2/3} < t \leq 1.
 \end{cases}
\end{equation}
\end{theorem}
\begin{proof}
For $b'<\sqrt{2/3}$, scale by $1/b$ and apply Theorem~\ref{thm:ball-interval} with $R=1/(b\sqrt2)$. At $b'=\sqrt{2/3}$, the displayed base equals $\sqrt2/b$, and the bound follows from a partition of $B_{1/\sqrt2}^n$ into sets of diameter less than $b$.
\end{proof}

Figure~\ref{fig:int} compares the bases in~\eqref{eq:2t1},
\eqref{eq:davies_for_intervals}, \eqref{eq:int_bl}, \eqref{eq:int_pros}, and~\eqref{eq:int_new}, with either $b$ or $t$ fixed. The Euclidean bounds apply after scaling by $1/b$. Depending on the parameters, any of~\eqref{eq:int_bl}, \eqref{eq:int_pros}, and~\eqref{eq:int_new} can give the smallest of these bases.

\section{Statement of AI use}
We used GPT-6.1 Sol to obtain the main steps of the proof of Theorem~\ref{main}. The authors independently verified the proof and simplified some steps.

\clearpage
\appendix
\section{Comparison of the bounds}

\begin{figure}[htbp]
    \centering
    \begin{tabular}{cc}
        \includegraphics[width=0.45\textwidth]{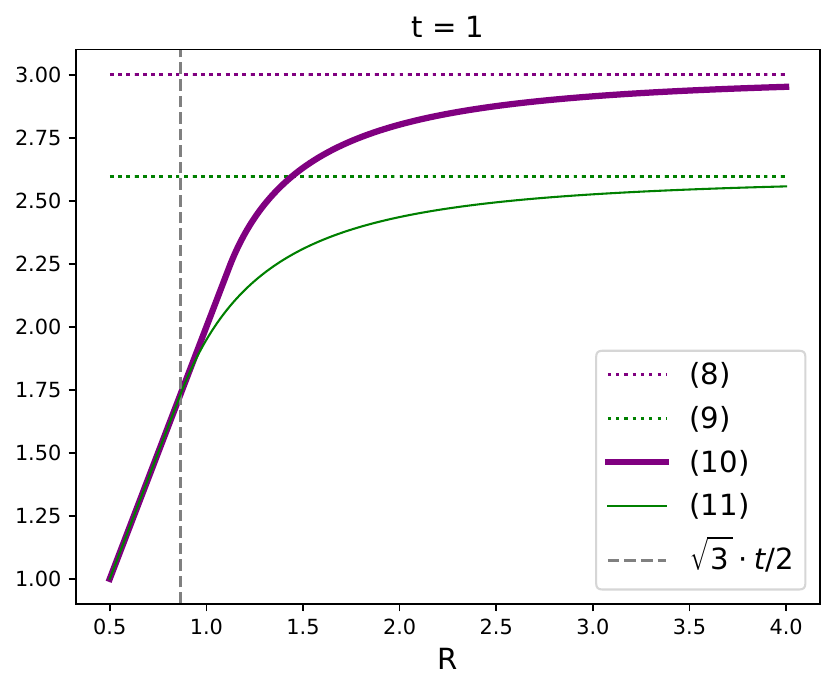} &
        \includegraphics[width=0.45\textwidth]{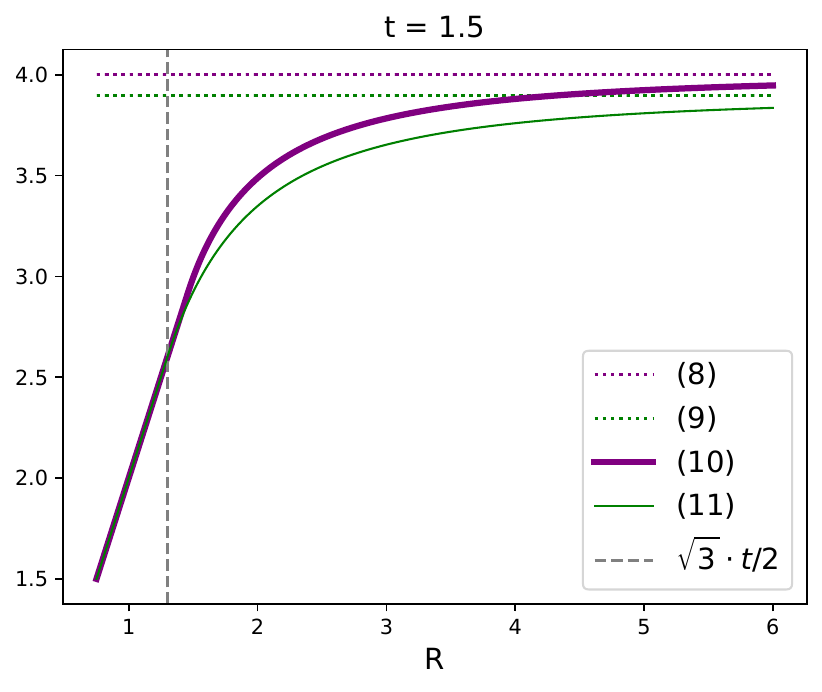} \\[2mm]
        \includegraphics[width=0.45\textwidth]{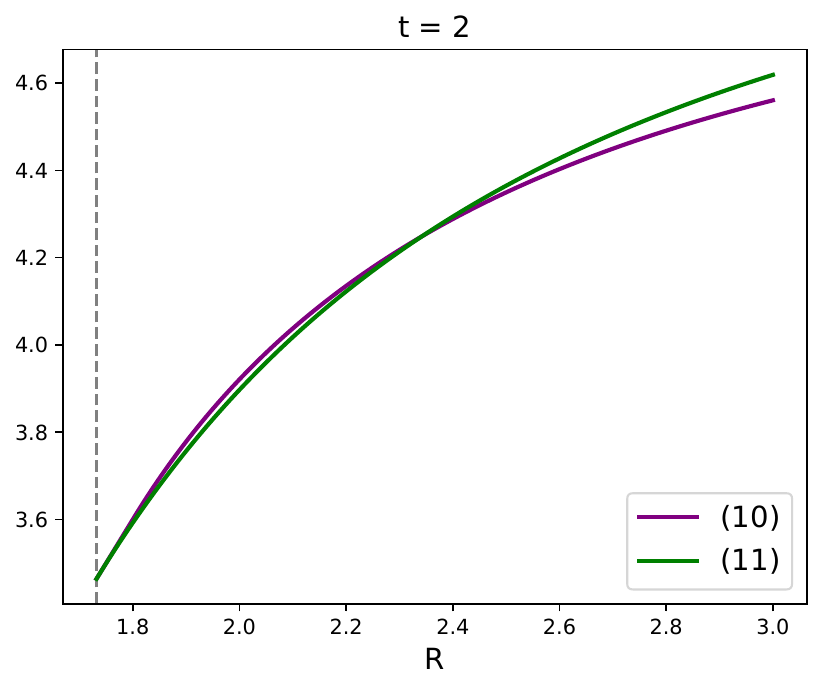} &
        \includegraphics[width=0.45\textwidth]{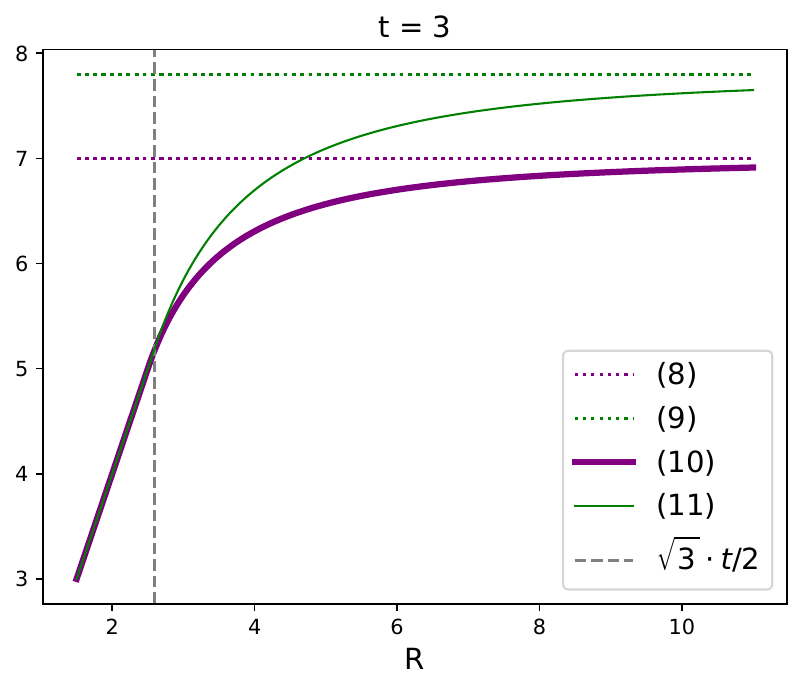} \\
    \end{tabular}
    \caption{Exponential bases in the upper bounds~\eqref{eq:2t1}, \eqref{eq:davies_for_intervals}, \eqref{eq:brnt1}, and~\eqref{eq:brnt2} for $\chi(B_R^n,[1,t])$. }
    \label{fig:int_balls}
\end{figure}

\begin{figure}[htbp]
\centering
\includegraphics[width=0.65\textwidth]{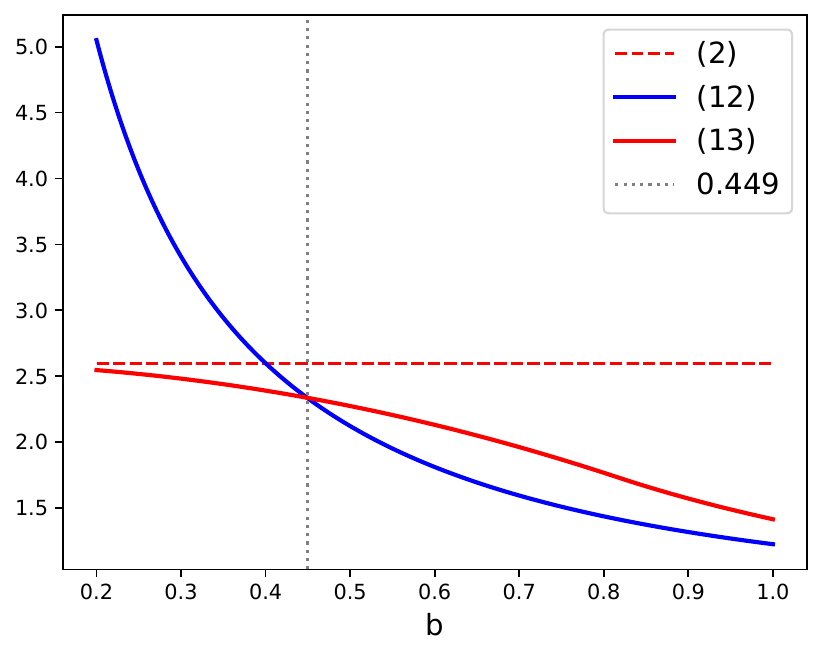}
\caption{Exponential bases in the upper bounds~\eqref{eq:new}, \eqref{eq:tilde_chi_n_b_1}, and~\eqref{eq:tilde_chi_n_b_2} for $\tilde\chi(n,b)$. }
\label{fig_1_dist}
\end{figure}

\begin{figure}[htbp]
    \centering
    \begin{tabular}{cc}
        \includegraphics[width=0.45\textwidth]{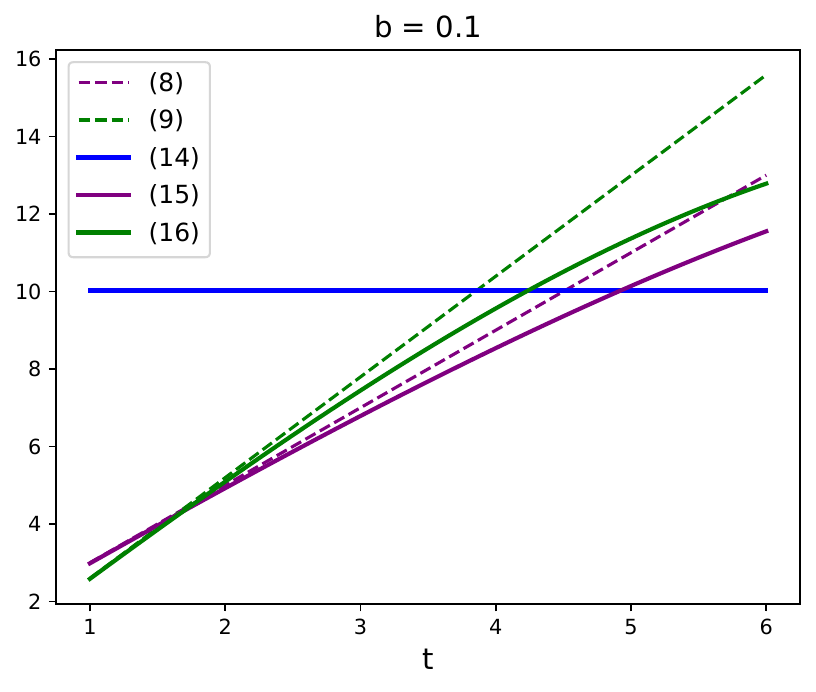} &
        \includegraphics[width=0.45\textwidth]{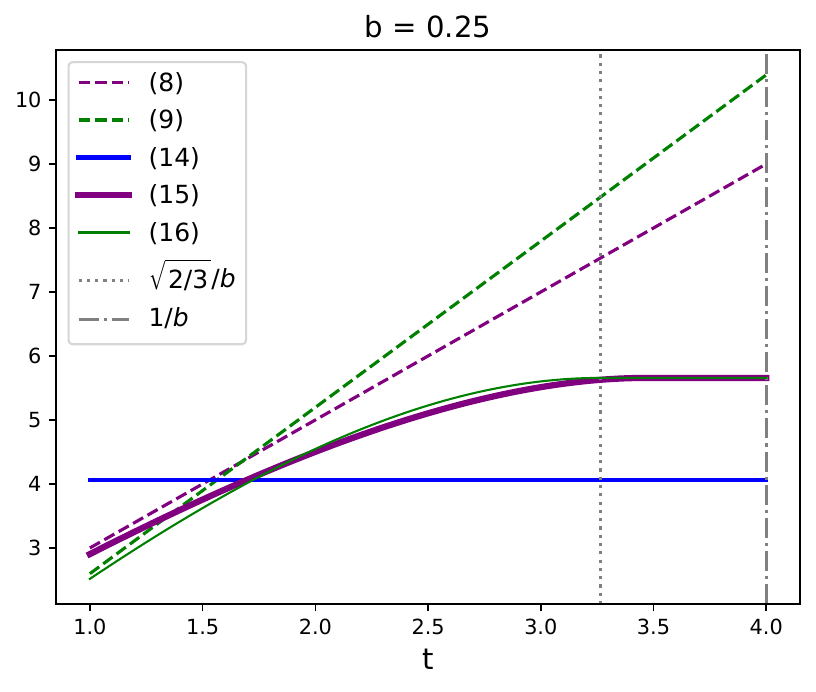}
        \\[2mm]
        \includegraphics[width=0.45\textwidth]{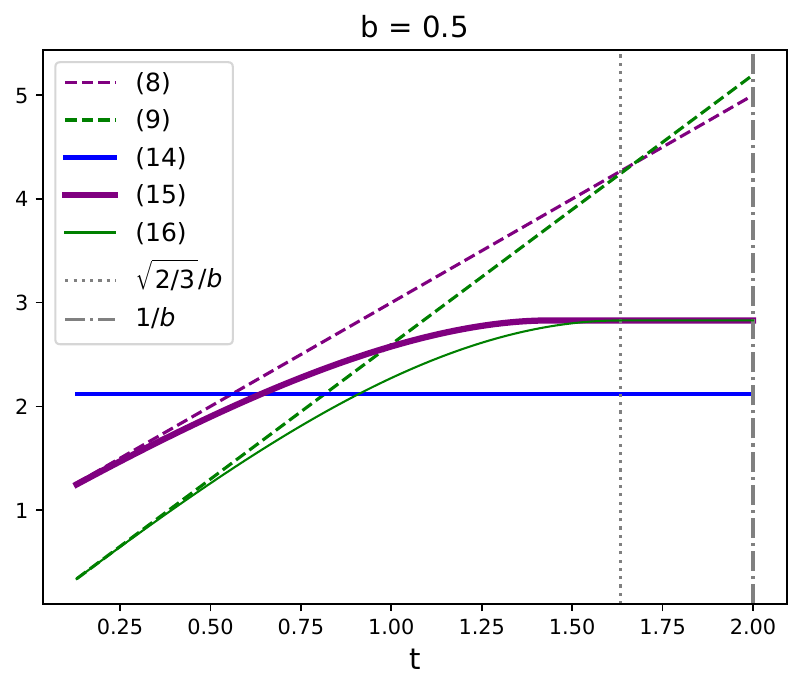}
        &
        \includegraphics[width=0.45\textwidth]{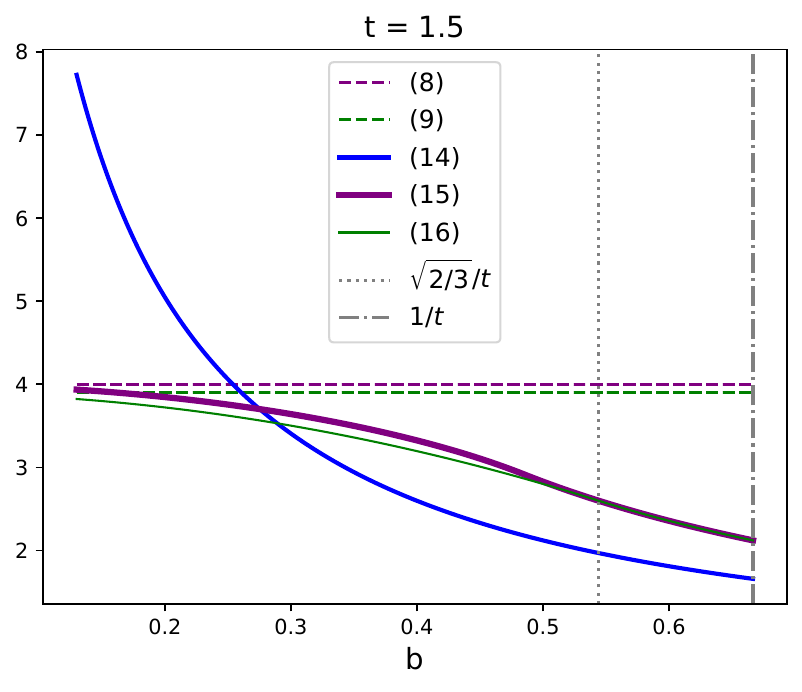}
        \\[2mm]
        \includegraphics[width=0.45\textwidth]{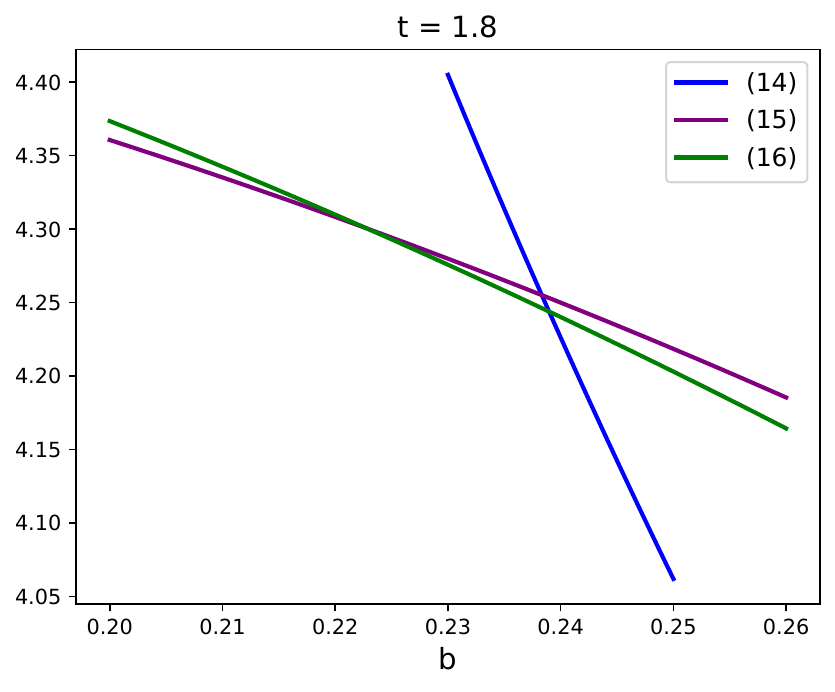}
        &
        \includegraphics[width=0.45\textwidth]{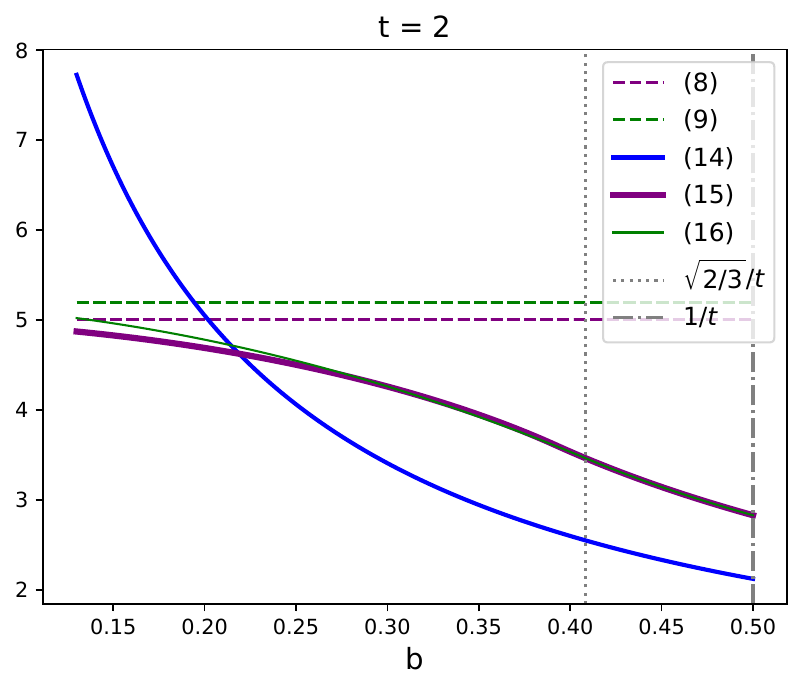}
        \\
    \end{tabular}
    \caption{Exponential bases in the upper bounds~\eqref{eq:2t1}, \eqref{eq:davies_for_intervals}, \eqref{eq:int_bl}, \eqref{eq:int_pros}, and~\eqref{eq:int_new} for $\tilde\chi(n,[b,bt])$.}
    \label{fig:int}
\end{figure}
\end{document}